\documentclass[11pt]{article}
\usepackage{epigamath}

\usepackage[notext]{kpfonts}
\usepackage{baskervald}

\setpapertype{A4}

\usepackage[english]{babel}

\usepackage[dvips]{graphicx}     

\title{Smoothing cones over K3 surfaces}
\titlemark{Smoothing cones over K3 surfaces}
\author{\vspace{0cm} Stephen Coughlan and Taro Sano}
\authoraddresses{
\authordata{Stephen Coughlan}{\firstname{Stephen} \lastname{Coughlan}\\
\institution{Institut f\"ur Algebraische Geometrie, Leibniz Universit\"at Hannover, 30167 Hannover, Germany \\
Current address: Mathematisches Institut, Lehrstuhl Mathematik VIII, Universit\"atsstrasse 30, 95447 Bayreuth, Germany}\\
\email{stephen.coughlan@uni-bayreuth.de}}\\
\authordata{Taro Sano}{\firstname{Taro}
\lastname{Sano}\\
\institution{Department of Mathematics, Faculty of Science, Kobe University, 1-1, Rokkodai-cho, Nada-ku, Kobe, 657-0029, Japan}\\
\email{tarosano@math.kobe-u.ac.jp}}
}
\authormark{S. Coughlan and T. Sano}
\date{\vspace{-5ex}} 
\journal{\'Epijournal de G\'eom\'etrie Alg\'ebrique} 
\acceptation{Received by the Editors on November 7, 2017, and in final form
on November 22, 2018. \\ Accepted on November 28, 2018.}

\newcommand{\articlereference}{Volume 2 (2018), Article Nr. 15}

\usepackage[all]{xy}

\allowdisplaybreaks

\newcommand{\A}{\mathbb A}
\newcommand{\C}{\mathbb C}
\renewcommand{\P}{\mathbb P}

\newcommand{\ZZ}{\mathbb Z}

\newcommand{\Oh}{\mathcal O}

\newcommand{\T}{\mathcal{T}}

\DeclareMathOperator{\coker}{coker}
\DeclareMathOperator{\Cliff}{Cliff}
\DeclareMathOperator{\Ext}{Ext}
\DeclareMathOperator{\Spec}{Spec}
\DeclareMathOperator{\Proj}{Proj}
\DeclareMathOperator{\Pic}{Pic}

\DeclareMathOperator{\Hilb}{Hilb}
\DeclareMathOperator{\Def}{Def}

\DeclareMathOperator{\Art}{Art}
\DeclareMathOperator{\Sets}{Sets}

\DeclareMathOperator{\Pf}{Pf}
\DeclareMathOperator{\Sing}{Sing}

\renewcommand{\thesubsection}{\arabic{section}.\arabic{subsection}}

\newcounter{subsectionnum}
\numberwithin{subsectionnum}{numsection}
\newlength{\lengthtitle}
\newcommand{\newsubsection}[1]{
\settowidth{\lengthtitle}{#1}
\ifnum\lengthtitle=0\paragraph{\bfseries \thesubsectionnum.}\else\paragraph{\bfseries \thesubsectionnum. #1.}\fi
\refstepcounter{subsectionnum}}

\newtheorem{thm}[subsectionnum]{Theorem}
\newtheorem{cor}[subsectionnum]{Corollary}
\newtheorem{lem}[subsectionnum]{Lemma}
\newtheorem{prop}[subsectionnum]{Proposition}

\newtheorem{rem}[subsectionnum]{Remark}
\newtheorem{eg}[subsectionnum]{Example}
\newenvironment{acknowledge}{\paragraph*{Acknowledgements.}}{}
\newenvironment{pf-thm}[1][\proofname]{\noindent\textit}{\hfill$\qed$}

\begin{document}


\maketitle



\begin{prelims}


\def\abstractname{Abstract}
\abstract{We prove that the affine cone over a general primitively polarised K3 surface of genus $g$ is smoothable if and only if $g\le10$ or $g=12$. We also give several examples of singularities with special behaviour, such as surfaces whose affine cone is smoothable even though the projective cone is not.}



\vspace{0.25cm}

\languagesection{Fran\c{c}ais}{%

\textbf{Titre. Sur la lissabilit\'e des c\^ones sur les surfaces K3} \commentskip \textbf{R\'esum\'e.} Nous montrons que le c\^one affine sur une surface K3 primitivement polaris\'ee g\'en\'erale de genre $g$ est lissable si et seulement si $g\leq10$ ou $g=12$. Nous exhibons \'egalement plusieurs exemples de singularit\'es affichant des comportements sp\'ecifiques, tels que des surfaces dont le c\^one affine est lissable alors m\'eme que le c\^one projectif ne l'est pas.}

\end{prelims}


\newpage

\setcounter{tocdepth}{1} \tableofcontents

\section{Introduction}

\newsubsection{}
In this paper, we study deformations and smoothability of the affine cone over a polarized manifold.
See \ref{subsect:basicproperty} for basic notions.

The cone over a normal elliptic curve is smoothable if and only if the curve has degree $\le 9$ \cite{Pinkham}, the cone over a projectively normal abelian variety of dimension $\ge 2$ is never smoothable \cite{Schl}, and the cone over a curve of genus $\ge2$ embedded in degree at least $4g-3$ is not smoothable, if the curve is not hyperelliptic, trigonal, or a plane quintic cf.~\cite[\S 15]{Stevens} and references therein. Aside from the case of elliptic curves, in all of the above situations, the only deformations are again cones.


The cone over a K3 surface is a natural 3-dimensional generalisation of the cone over an elliptic curve; it is a normal, Gorenstein, isolated log canonical singularity. One of the main results of this present work is the following theorem.
\begin{thm}\label{thm!main} Let $S$ be a general K3 surface with primitive polarisation of genus $g$. Then the affine cone over $S$ is smoothable if and only if $g\le 10$ or $g=12$. Indeed, if $g=11$ or $g\ge13$ then the only deformations of the affine cone over $S$ are conical.
\end{thm}

A deformation of the affine cone is called \emph{conical} when the conclusion of Proposition \ref{prop!weighted-schlessinger-conditions}(ii) holds.
Cones over non-general K3 surfaces of genus $g=11$ or $g>12$ may still be smoothable (see \S\ref{sec!smoothing-fanos}).

\newsubsection{} The theorem is proved in Section \ref{sec!K3-cones}. The ``only if'' part follows from the vanishing of all graded parts of $T^1$ of the affine cone which have non-zero degree.
It is proved by using a deep theorem of Beauville \cite{Beauville} and its slight modification (see \ref{sec!vanishing-k-ge-1}). A weaker result can be derived from Green's conjecture, but with a precise condition on the polarization (see \ref{thm!keq2}). The ``if'' part is proved by sweeping out the cone because for $g\le10$ and $g=12$, the projective cone over $S$ deforms to a smooth Fano 3-fold, (see \ref{sec!sweeping}).


\newsubsection{}\label{sec!intro-projective-smoothings} 
Pinkham \cite{Pinkham2} gave an example of a $0$-dimensional variety whose affine cone is smoothable, even though the projective cone is not. The cone is Cohen--Macaulay but not Gorenstein or normal. In section 2, we prove the following:
\begin{thm}\label{thm!proj-cone} There exists a smooth, projectively normal surface $S$ such that the affine cone over $S$ is smoothable but the projective cone is not.
\end{thm}

The example is a particular surface of general type in its canonical model. We do not know of any example where $S$ is a K3 surface. In light of Pinkham's theorem on elliptic curves and Theorem \ref{thm!main} above, we ask:
\begin{quote} If $S$ is a K3 surface, is the affine cone over $S$ smoothable if and only if the
projective cone is
smoothable?\end{quote}
By \cite{CLM,CLM2}, the projective cone over a general K3 surface of Picard rank 1 is smoothable if and only if $g\le10$ or $g=12$.

\newsubsection{}\label{sec!intro-examples} 
We also construct K3 surfaces whose affine cone has several smoothing components:
\begin{thm}\label{thm!two-smoothings} There exist primitively polarised K3 surfaces $(S,L)$ of genus $7$, such that $C_a(S,L)$ has at least two topologically distinct smoothings.
\end{thm}

The proof is in \S\ref{sec!smoothing-fanos}, along with an analysis of cones over imprimitively embedded K3 surfaces, and cones over special K3 surfaces of large genus.

\newsubsection{}
Given a very ample line bundle $L$ on a smooth projective variety $V$ which induces a projectively normal embedding $V \hookrightarrow \P^N$,
we have the ``classical'' projective cone $C_p(V) \subset \P^{N+1}$ and  affine cone $C_a(V) \subset \A^{N+1}$.
Pinkham (\cite[Theorem 5.1]{Pinkham}) showed that, if the eigenspace $T^1_{C_a(V)}(k) =0$ for $k >0$, then the restriction homomorphism
 $\Hilb_{C_p(V) \subset \P^{N+1}} \rightarrow \Def_{C_a(V)}$ is formally smooth, where $\Hilb_{C_p(V) \subset \P^{N+1}}$ is the Hilbert functor and $\Def_{C_a(V)}$
 is the usual deformation functor.  Moreover, Schlessinger (\cite[\S 4.3]{Schl}) showed that, if $T^1_{C_a(V)}(k) =0$ for $k \neq 0$,
 then we can define a morphism $\Hilb_{V \subset \P^N} \rightarrow \Def_{C_a(V)}$ and it is formally smooth.

 In Section \ref{sec!weighted-homog}, we generalise these results to the case where $L$ is only assumed to be ample. This is probably known to the experts, but a proof has not been written down, so we give one in \ref{prop!weighted-schlessinger-conditions}. As an application, in \ref{cor:abeliancone}, we show:
 \begin{thm} The affine cone over any polarised abelian variety of dimension $\ge2$ has only conical deformations.
 \end{thm}

\newsubsection{}


We work over the complex numbers unless otherwise stated.

\begin{acknowledge} We thank Paul Hacking, Yoshinori Namikawa and Miles Reid for useful discussions. We also thank Angelo Lopez and Ciro Ciliberto for pointing out \cite{CLM,CLM2} after the first version of this article appeared, and also to Shigeru Mukai for helpful comments on the first version. Finally, we thank the referees for several helpful comments, corrections and advice. SC was supported by the DFG through grant Hu 337-6/2, and
ERC Advanced grant no.~340258, TADMICAMT. TS was supported by Max Planck Institute for Mathematics, JSPS Research Fellowships for Young Scientists and JST tenure track program.
\end{acknowledge}

\section{Affine and projective cones}

\newsubsection{Basic properties of cones}\label{subsect:basicproperty}

We use standard notation from deformation theory, see for example \cite{Hartshorne-deformation}, \cite{Sernesi}.
Let $X$ be an algebraic scheme. A \emph{deformation} of $X$ over a scheme $B=\Spec A$ of finite type with a closed point $0 \in B$  is a flat morphism $\pi\colon\mathcal{X}\to B$ together with a closed immersion $X \hookrightarrow \mathcal{X}$ which identifies $X$ with the closed fibre over $0$.
A deformation is called {\it infinitesimal} if $A$ is local Artinian.  We say that $X$ is \emph{smoothable} if
there exists a deformation $\pi\colon\mathcal{X}\to B$ of $X$ over an integral scheme $B$ of finite type
whose fibre $\mathcal{X}_b$ is smooth for general $b\in B$ (cf. \cite[\S 29]{Hartshorne-deformation}).

Let $(X,L)$ be a {\it polarised manifold}, that is, $X$ is a smooth projective variety such that $\dim X \ge 1$ and $L$ is an ample line bundle.
Let
\[
R(X,L):= \bigoplus_{k \ge 0} H^0(X, L^{\otimes k}).
\]
The {\it affine cone} over $(X,L)$ is $C_a(X,L):=\Spec R(X,L)$ and the {\it projective cone} over $(X,L)$ is $C_p(X,L):=\Spec R(X,L)[x]$, where $x$ has degree $1$.
By \cite[8.8.6]{EGA2}, $C_a(X,L)$ is normal.
We recall the following property.

\begin{prop}\label{prop:CMGoren}
Let $(X, L)$ be a polarised manifold such that $\dim X \ge 1$.
Then we have the following:
\begin{itemize}
\item[\rm (i)] The cone $C_a(X,L)$ is Cohen--Macaulay if and only if $H^i(X, L^{\otimes k}) =0$ for all $0<i< \dim X$ and $k \in \mathbb{Z} $.
\item[\rm (ii)] The cone $C_a(X,L)$ is Gorenstein if and only if it is Cohen--Macaulay and $\omega_X \simeq L^{\otimes m}$ for some $m \in \mathbb{Z}$.
\end{itemize}
\end{prop}

\begin{proof}
\noindent For (i), it is enough to check the conditions (a) and (b) in
\cite[5.1.6(ii)]{GW}.
We can check (a) by the construction of $C_a(X,L)$.
The condition (b) is nothing but our assumption. Part (ii) follows from \cite[5.1.9]{GW}.
\qed
\end{proof}

%

\newsubsection{A smoothable affine cone with a non-smoothable projectivization}


The following example proves Theorem \ref{thm!proj-cone}.

\begin{eg}\label{ex!non-proj-smoothing}{\rm
Let $S$ be a divisor of bidegree $(3,4)$ in $\P:=\P^1\times\P^2$. If $S$ has at worst ordinary double points, then $S$ is a regular surface of general type with $p_g=6$ and $K^2=11$. Indeed, by adjunction, $\omega_S=\Oh_S(1,1)$. From the standard short exact sequence $0\to \mathcal{I}_S\to \Oh_\P\to\Oh_S\to0$ and vanishing of $H^1(\Oh_\P(-2,-3))$, it follows that $p_g(S)=h^0(\Oh_{\P}(1,1))=6$. Similarly, $q(S)=h^1(\Oh_{\P})=0$ so $S$ is regular. Writing $H_1,H_2$ for the generators of $\Pic\P$, we compute $K_S^2=(3H_1+4H_2)(H_1+H_2)^2=11H_1H_2^2=11$.

The above discussion shows that the canonical model of $S$ is induced by the Segre embedding of $\P^1\times\P^2$ in $\P^5$. We next describe the defining equations of the canonical model. Let $s_1,s_2$, $t_1,t_2,t_3$ be the homogeneous coordinates on $\P^1\times\P^2$ and let $F\in H^0(\Oh_\P(3,4))$ be the defining equation of $S$. We choose $f_1,f_2,f_3\in H^0(\Oh_{\P}(3,3))$ such that $F=t_1f_1+t_2f_2+t_3f_3$. Then the coordinates giving the Segre embedding are $x_1=s_1t_1,x_2=s_1t_2,x_3=s_1t_3,x_4=s_2t_1,x_5=s_2t_2,x_6=s_2t_3$ and $f_1,f_2,f_3$ may be written as cubics in the $x_i$. The canonical model of $S$ in $\P^5$ is thus defined by the following five equations:
\begin{gather*}
\Pf_5= x_2x_4 -x_1x_5,\ \Pf_4= x_3x_4 -x_1x_6,\ \Pf_3= x_3x_5- x_2x_6,\\
\Pf_2=s_1F=x_1f_1+x_2f_2+x_3f_3,\ \Pf_1=s_2F=x_4f_1+x_5f_2+x_6f_3,
\end{gather*}
where $\Pf_2$, $\Pf_1$ are obtained by writing $s_1F$ (respectively $s_2F$) in terms of the $x_i$. According to the Buchsbaum--Eisenbud theorem on Gorenstein codimension 3 ideals, these equations may be written as $4\times4$ Pfaffians of the skew matrix $M$
of the form
\[
M=
\begin{pmatrix}
0 & 0 & x_1 & x_2 & x_3 \\
0 & 0 & x_4 & x_5 & x_6 \\
-x_1 & -x_4 & 0 & f_3 & -f_2 \\
 -x_2 & -x_5 &-f_3 & 0 & f_1 \\
 -x_3 & -x_6 & f_2 & -f_1 & 0
  \end{pmatrix},\]
where $\Pf_i$ is the Pfaffian of the skew symmetric matrix $M_i$ obtained from $M$
by deleting $i$-th row and column.
The first three equations define the Segre embedding, and the last two cut out the divisor $S$.

Let $X=C_a(S,K_S)\subset\A^6$ be the affine cone over the canonical model of $S$. Then by construction, $X$ is a Gorenstein normal 3-dimensional singularity. The equations defining $X$ are still the $4\times 4$ Pfaffians of $M$, and the coordinates on $\A^6$ are $x_1,\dots,x_6$.

All deformations of $X$ are obtained by varying the entries of $M$ \cite{KL,Waldi} or \cite[Theorem 9.7]{Hartshorne-deformation}. Thus after coordinate changes, the general fibre $X'$ of any deformation of $X$ is defined by the Pfaffians of
\[M'=
\begin{pmatrix}
0 & g & x_1 & x_2 & x_3 \\
-g & 0 & x_4 & x_5 & x_6 \\
-x_1 & -x_4 & 0 & f_3' & -f_2' \\
 -x_2 & -x_5 &-f_3' & 0 & f_1' \\
 -x_3 & -x_6 & f_2' & -f_1 '& 0
  \end{pmatrix},\]
where $f'_i=f_i+h_i$ for some polynomials $h_i$, and $g$ is an arbitrary polynomial.
Then the $4\times4$ Pfaffians of $M'$ are
\begin{gather*}
\Pf_5=g f_3'+x_2x_4-x_1x_5,\ \Pf_4=-g f_2'+ x_3x_4 -x_1x_6,\ \Pf_3=g f_1'+ x_3x_5 -x_2x_6 , \\
\Pf_2=x_1f_1' + x_2f_2' + x_3 f_3',\ \Pf_1=x_4 f_1' + x_5 f_2' + x_6 f_3'.
\end{gather*}

The smoothability of  $X$ is well known (cf.~\cite[Section 5]{KL}). Let $g$ be a nonzero constant, and choose $h_i$ sufficiently general with some terms of degree $\le1$. Since $g$ is constant, Pfaffians 1 and 2 are redundant, and $X'$ is a nonsingular complete intersection for suitably chosen $h_i$.

Now restrict to deformations $X'$ that are induced by a deformation of the projective cone $C_p(S, K_S)\subset\P^6$. Then $g\equiv0$ for degree reasons, and $h_i$ must have degree $\le3$ --- in particular, the above smoothing is not induced by $C_p(S, K_S)$. Since $g=0$, $X'$ passes through the origin, and a computation of the partial derivatives of Pfaffians 3, 4 and 5 shows that the Jacobian matrix of $X'$ must have rank $\le2$ there. Thus $X'$ is singular at the origin. As pointed out by the referee, an analysis of the tangent cone shows that, at best, the singularity of $X'$ is given by taking two hyperplane sections through the vertex of the cone over the Segre embedding of $\P^1\times\P^2$. These hyperplane sections are defined by some perturbations of $\Pf_1$ and $\Pf_2$ respectively.}
\end{eg}

\begin{rem}
{\rm For any $k\ge3$, we get 3-fold singularities with similar properties by taking a divisor $S_k$ in $\P^1\times\P^2$ of bidegree $(k,k+1)$.}
\end{rem}

\section{Proof of Theorem \ref{thm!main}}\label{sec!K3-cones}
In this section we prove
\begin{thm}\label{cor!neg} Let $S$ be a general K3 surface with primitive polarisation $L$ of genus $g$ and write $X=C_a(S,L)$. Then $T^1_X$ is concentrated in degree 0 if and only if $g=11$ or $g\ge 13$, and $X$ is smoothable if and only if $g\le10$ or $g=12$.
\end{thm}

Indeed, if $T^1_X(k)=0$ for all $k\ne0$, then $X$ has only conical deformations by Schlessinger \cite[Theorem 12.1]{Artin} (cf.~Proposition \ref{prop!weighted-schlessinger-conditions}).

\newsubsection{}\label{sec!sweeping} A Fano 3-fold with $b_2=1$ and genus $g$ exists when $2 \le g\le10$ or $g=12$ (cf.~\cite[\S4]{Mukai}). Then by \cite[Corollary 4.1]{Beauville}, a general primitively polarized K3 surface $(S, L)$ is obtained from $S \in |{-K_W}|$ for $W$ a Fano 3-fold with $b_2=1$
and $L= -K_W|_S$. 
Let $\sigma \in H^0(W, -K_W)$ be the defining section of $S$.
Then we may regard $C_a(S, L)$ as a divisor in $C_a(W, -K_W)$. 
Now let $\mathcal{X} \subset C_a(W, -K_W) \times \mathbb{A}^1$ be the zero locus of $\sigma + \lambda$,
where $\lambda$ is the parameter of the affine line $\mathbb{A}^1$.
This induces a smoothing $\mathcal{X} \rightarrow \mathbb{A}^1$ of $X$, which is called \emph{sweeping out the cone}.


\newsubsection{Computing graded pieces of $T^1_X$}\label{sec!computing-T1}
Let $(V,L)$ be a polarized manifold. By \cite{Schl,Pinkham}, the $\C^*$-action on $X=C_a(V,L)$ induces a grading on $T^1_X$, the space of isomorphism classes of first order infinitesimal deformations of $X$:
\[T^1_X=\bigoplus_{k\in\ZZ} T^1_X(k).\]
By \cite[Theorem 3.7]{Wahl-equising} we have
\[T^1_X(k)\subset H^1(V,\mathcal{E}_L\otimes L^{\otimes k}),\]
with equality when $H^1(V,L^{\otimes k})=0$ for all $k$ in $\ZZ$, where $\mathcal{E}_L$ is the extension
\begin{equation}\label{eqn:E_Ldefn}
0 \rightarrow \mathcal{O}_V \rightarrow \mathcal{E}_L
\rightarrow \mathcal{T}_V \rightarrow 0
\end{equation}
corresponding to $c_1(L) \in H^1(V, \Omega^1_V) \cong \Ext^1(\mathcal{T}_V, \mathcal{O}_V)$. When $V=S$ is a polarized K3 surface, $H^1(S,L^{\otimes k})=0$ for all $k$ in $\ZZ$, and so $T^1_X(k)\cong H^1(S,\mathcal{T}_S\otimes L^{\otimes k})$.
\newsubsection{Vanishing for $|k|\ge2$}\label{sec!vanishk2}
We recall the following criterion of Wahl for vanishing of $T^1(k)$:
\begin{thm} {\rm (Wahl \cite[Corollary 2.8]{Wahl})} \ Suppose the free resolution of $\Oh_S$ begins with
\begin{equation}\label{eqn!wahl}
\Oh_S\gets\Oh_{\P}\gets\Oh_{\P}(-2)^a\gets\Oh_{\P}(-3)^b\gets\dots.
\end{equation}
Then $T^1_{C_a(S)}(k)=0$ for $k\le -2$.
\end{thm}

\begin{thm}\label{thm!keq2} Let $S$ be a K3 surface with primitive polarisation $L$ of Clifford index $>2$. Let $X$ be the affine cone over $(S, L)$, then $T^1_X(k)$ vanishes for $|k|\ge 2$.
\end{thm}
\begin{proof}
By \cite{SD}, we can choose $C\in|L|$ a nonsingular irreducible curve. Since $C$ is a hyperplane section of $S\subset\P^g$ and the coordinate ring of $S$ is Gorenstein,
the free resolution of $\Oh_S$ is inherited from that of $\Oh_C$.
According to Green's conjecture \cite{Green}, the resolution of $\Oh_C$ has the form required by Wahl's criterion if and only if $\Cliff C>2$.
Since Green's conjecture holds for canonical curves on any K3 surface by Voisin \cite{Voi2,Voi5} and Aprodu--Farkas \cite{AF}, the theorem is proved.
\qed
\end{proof}

\newsubsection{Vanishing for $|k|=1$}\label{sec!vanishing-k-ge-1}

We recall the following theorem of Beauville and Mori--Mukai.

\begin{thm} {\rm (cf.~Beauville {\cite[\S 5.2]{Beauville}}, Mukai {\cite[\S 4]{Mukai}})}
\label{thm!vanishing} \ Let $S$ be a general K3 surface with primitive polarisation
$L$ of genus $g=11$ or $g\ge13$. Then $H^1(S,\Omega^1_S\otimes L)=0$.
\end{thm}

Since $T^1_X(-1)\cong H^1(S,\Omega_S^1\otimes L)$ and $T^1(k)\cong T^1(-k)$ because $S$ is a K3 surface, we have
the required vanishing.

We briefly explain the proof of Theorem \ref{thm!vanishing}.
Let $\mathcal{P}_g$ be
the moduli stack of pairs $(S,C)$, where $(S,L)$ is a primitively polarized K3 surface with $c_1(L)^2 = 2g-2$,
and $C$ is a stable curve in $|L|$.
Beauville \cite[(5.1)]{Beauville} shows that the vanishing
in Theorem \ref{thm!vanishing} is equivalent to generic finiteness of the forgetful morphism of
smooth irreducible Deligne--Mumford stacks
$\varphi_{g} \colon \mathcal{P}_{g} \rightarrow \overline{\mathcal{M}}_{g}$ defined by $(S,C)\mapsto C$. 
Mori and Mukai \cite{MM2}, \cite[Theorem 7]{Mukai} prove that $\varphi_g$ is generically finite
when $g=11$ and $g\ge 13$ by constructing explicit pairs $(S,C)$ for which the
fibre $\varphi_g^{-1}(C)$ is finite.

\section{Cones over some special K3 surfaces}\label{sec!smoothing-fanos}

In this section we examine the behaviour of $C_a(S,L)$ in some situations where $S$ is a non-general K3 surface.

\newsubsection{Cones over imprimitively polarized K3 surfaces}\label{sec!imprimitive}

{\ }

\begin{prop} Let $(S,L)$ be a general primitively polarized K3 surface of genus $g$ and fix an integer $n>1$. Then $X=C_a(S,L^{\otimes n})$ is smoothable if and only if one of the following holds:
$2\le g\le 6$ and $n=2$, or $g=3,4$ and $n=3$, or $(g,n)=(3,4)$.
\end{prop}
\begin{proof}
First note that $T^1_X(k)\cong T^1_Y(kn)$, where $Y=C_a(S,L)$. Since the Clifford index of $L$ is $\le\lfloor\frac{g-1}2\rfloor$, it follows from Green's conjecture and explicit computations for $g\le6$, that $T^1_X(k)$ vanishes for all $k\ne0$ when $(g,n)$ lies outside the stated values. For the converse, if $(g,n)\ne(3,3)$ then the smoothing is given by sweeping out the cone in the Fano 3-fold $(W,-K_W)$ with $-K_W=nA$, chosen so that $S\in|{-K_W}|$ and $A|_S=L$. In the special case $(g,n)=(3,3)$, the 3-fold $W_4\subset\P(1,1,1,1,3)$ inducing the smoothing of $X$ has a quotient singularity.
\qed
\end{proof}

\newsubsection{The cone over a K3 surface with $g=11$ or $g\ge13$ can nevertheless be smoothable}
If $W$ is a Fano 3-fold and $S$ is an anticanonical section of $W$ with polarization $\Oh_S(1):= -K_W|_S$,
then $C_a(S,\Oh_S(1))$ is smoothable.
From the Mori--Mukai classification \cite{MM} of Fano 3-folds with $b_2\ge2$, we see that such $S,W$ exist for $g=11$, $13\le g\le 29$ and $g=32$.
The case $g=33$ also occurs (see \ref{sec!imprimitive} with $(g,n)=(3,4)$). If $g>33$, then any smoothing of $C_a(S,\Oh_S(1))$ does not lift to the projective cone $C_p(S,\Oh_S(1))$. In spite of Example \ref{ex!non-proj-smoothing}, we expect that $C_a(S,\Oh_S(1))$ is not smoothable for \emph{any} $S$ of genus $>33$.


\newsubsection{K3 surfaces whose affine cone has at least two distinct smoothings}\label{sec!two-smoothings}

In this section, we prove Theorem \ref{thm!two-smoothings}. First recall the following example:

\begin{eg}\label{rem!dp6} {\rm The degree 6 del Pezzo surface $Y$ is a hyperplane section of $V_1=V\colon(1,1)\subset\P^2\times\P^2$ and $V_2=\P^1\times\P^1\times\P^1$. Thus $C_a(Y,{-K_Y})$ has two distinct smoothings. }
\end{eg}

Inspired by this, we found the following example:

\begin{eg}\label{ex!K3-double-dp6}{\rm
Let $\pi\colon S\to Y$ a double cover of the degree 6 del Pezzo surface $Y$, branched in $B\in|{-2K_Y}|$.
Let $L:= \pi^*(-K_Y)$ so that $(S, L)$ is a primitively polarised K3 surface of degree $12$  in $\P^7$.
By Example \ref{rem!dp6}, $Y=V_i\cap H_i$ for some $H_i\in|{-\frac12K_{V_i}}|$. Take $\pi_i\colon W_i\to V_i$ a double cover branched in $X_i\in|{-K_{V_i}}|$, where $X_i$ are chosen so that $X_i\cap H_i=B$ since $H^0(V_i, -K_{V_i}) \rightarrow H^0(Y, -2K_Y)$ is surjective. The $W_i$ are Fano 3-folds with distinct topology. Indeed, $W_1$ (respectively $W_2$) is number 2.6b (resp.~3.1) of the classification \cite{MM}. Moreover,
$W_i\cap\pi_i^*H_i=S$, so the affine cone $C_a(S,\Oh_S(1))$ is a hyperplane section of $C_a(W_i,-K_{W_i})\subset\A^{9}$ for each $i$, and so $C_a(S,\Oh_S(1))$ has two topologically distinct smoothings.}
\end{eg}

\section{On quasihomogeneous cones}\label{sec!weighted-homog}

Let $X$ be a projective manifold polarised by an ample line bundle $L$. We generalise Pinkham and Schlessinger's criteria on $T^1(k)$ \cite[Theorem 12.1]{Artin}, 
to the case where $L$ is not necessarily very ample. Choose generators $x_1,\dots,x_n$ of degrees $w_1,\dots,w_n$ for $R(X,L)=\bigoplus_{k\ge 0}H^0(X,L^{\otimes k})$ and let $\bar X=\Proj R(X,L)$ be the image of $X$ in weighted projective space $\P(w_1,\dots,w_n)$.

\begin{lem} The image $\bar X$ in $\P(w_1,\dots,w_n)$ is nonsingular and avoids the singular locus of $\P(w_1,\dots,w_n)$.
\end{lem}

\begin{proof} The isomorphism $X \simeq \Proj R(X,L)$ is elementary and $X$ is embedded into
$\P(w_1, \ldots, w_n)$ by the surjection $\C [z_1, \ldots, z_n] \rightarrow R(X, L)$ sending
$z_i$ to $x_i$ for $i=1, \ldots, n$.

Assume that $X \cap \Sing \P \neq \emptyset$. Then there exists $I:= \{ i_1, \ldots, i_l \} \subset \{1, \ldots, n \}$ such that $w_I:= \gcd (w_{i_1}, \ldots, w_{i_l})$ $>1$ and the corresponding stratum $\Pi_I\subset\Sing\P$ of index $w_I$
satisfies $X \cap \Pi_I \neq \emptyset$.
Let $m >0$ be a sufficiently large integer such that $L^{\otimes m}$ is very ample and
$\gcd(m, w_I) =1$. Since we have a surjection $\C [z_1, \ldots, z_n] \rightarrow R(X,L)$ and
it induces a surjection $H^0(\P, \Oh_{\P}(m)) \rightarrow H^0(X, L^{\otimes m})$,
we obtain $V(s) \supset \Pi_I \cap X$ for nonzero $s \in H^0(X, L^{\otimes m})$.
This contradicts the base point freeness of $|L^{\otimes m}|$.
\qed
\end{proof}



Let $\Art_{\C}$ be \,the \,category \,of \,Artinian \,local $\C$-algebras \,with \,residue \,field $\C$. \,
We \,denote \,by $\Hilb^w_Y\colon {\Art_{\mathbb{C}}} \rightarrow (\Sets)$, the weighted Hilbert functor parametrizing embedded deformations of $Y\hookrightarrow \P^n(w)$  in the weighted projective space $\P^n(w)$.

\begin{prop}\label{prop!weighted-schlessinger-conditions}
Let $X$ be a projective manifold, $L$ an ample line bundle on $X$ and $X \hookrightarrow \P(w_1,\dots,w_n)$ be the embedding
determined by generators $x_1, \ldots, x_n \in R(X, L)$.
\begin{itemize}
\item[\rm (i)] (Negative graded case) Suppose that $T^1_{C_a}(k)=0$ for all $k > 0$.
Then the restriction map
\[
\Phi\colon\Hilb^w_{C_p(X, L)} \rightarrow \Def_{C_a(X,L)}
\]
is formally smooth.
\item[\rm (ii)]  (Conical deformations)
Suppose that $T^1_{C_a}(k)=0$ for all $k \neq 0$.
Then we have a canonical morphism of functors
\[
\Psi\colon\Hilb^w_{X} \rightarrow \Def_{C_a(X,L)}
\]
and it is formally smooth, that is,  $C_a(X,L)$ has only conical deformations.
\end{itemize}
\end{prop}


A weaker version of part (i) can be extracted from \cite{Pinkham}: the restriction map $\Phi\colon\Hilb^w_{C_p(X, L)} \rightarrow \Def_{C_a(X,L)}$ has a section. Indeed, Pinkham and Schlessinger \cite[Proposition 2.3]{Pinkham} showed that a quasihomogeneous cone has a versal deformation,
and a small modification of the argument used in \cite[Theorem 4.2]{Pinkham} shows the claim. 

\begin{proof}
We generalise the approach of \cite[Theorem 12.1]{Artin} to the weighted setting. For part (i), we need to show that the following two properties hold:

\begin{enumerate}
\item[(1)] $d\Phi\colon\Hilb^w_{C_p}(k[\epsilon])\to \Def_{C_a}(k[\epsilon])$ is surjective.
\item[(2)] Define $\xi_A:= C_{p,A} \in \Hilb^w_{C_p}(A)$. Let $\bar{\xi}_A:= \Phi (\xi_A) \in \Def_{C_a}(A)$
be its image and assume that $\bar{\xi}_A$ can be lifted over a small extension $A' \in \Art_{\mathbb{C}}$
of $A$. Then there exists a lift $\xi_{A'} \in \Hilb^w_{C_p}(A')$ of $\xi_A$ over $A'$.
\end{enumerate}

We prove (1). Let $C_p\subset\P(1,w_1,\dots,w_n)$ denote the projective cone over $X$ and let $\pi_p\colon C'_p\to X$ be the $\C$-bundle over $X$ arising from the punctured projective cone. Since $C_p$ is normal at the vertex, we have
\begin{gather*}
\Hilb^w_{C_p}(k[\epsilon])=H^0(N_{C_p/\P(1,w)})=H^0(C'_p,\pi_p^*N_{X/\P(w)})=
H^0(X,{\pi_p}_*\pi_p^* N_{X/\P(w)})\\=\bigoplus_{j\ge0} H^0(X,N_{X/\P(w)}\otimes L^{\otimes -j}).
\end{gather*}
Now, $\Def_{C_a}(k[\epsilon])=T^1_{C_a}$, and according to \cite[Theorem 3.7]{Wahl-equising}, the graded pieces of $T^1_{C_a}$ are
\[T^1_{C_a}(k)=\coker\Big(H^0(X,\bigoplus_{i=1}^n L^{\otimes (k+w_i)})\to H^0(X,Q\otimes L^{\otimes k})\Big)\]
where $Q$ is the cokernel of $\mathcal{E}_L\to\bigoplus_i L^{\otimes {w_i}}$. Now, $Q$ is simply the normal
bundle to $X$ in $\P(w_1,\dots,w_n)$.
Since $T_{C_a}^1(k)=0$ for $k>0$, we see that
\[d\Phi\colon\bigoplus_{k\ge 0}H^0(N_{X/\P(w)}\otimes L^{\otimes-k})\to T^1_{C_a}=\bigoplus_{k\ge 0}T^1_{C_a}(-k)\]
is surjective.

Next we prove (2). The obstruction to lifting $\bar{\xi}_{A'}$ to $\xi_{A'}$ lives in $H^1(C_p,N_{C_p/\P(1,w)})$. As before, since $C_p$ is normal at the vertex, we have an inclusion $H^1(C_p,N_{C_p/\P(1,w)})\subset \bigoplus_{k\ge 0} H^1(X,N_{X/\P(w)}\otimes L^{\otimes -k})$. Thus we have an inclusion
\[H^1(C_p,N_{C_p/\P(1,w)})\subset H^1(C'_a,N_{C'_a})=\bigoplus_{k=-\infty}^{\infty} H^1(X,N_{X/\P(w)}\otimes L^{\otimes k}),\]
where $C_a'$ denotes the punctured affine cone, which is a $\C^*$-bundle over $X$. By assumption, $\bar \xi_A|_{C'_a}$ lifts to $\bar \xi_{A'}|_{C'_a}$ so the image of this obstruction in $H^1(C'_a,N_{C'_a})$ vanishes. This implies the existence of $\xi_{A'}$ in $\Hilb^w_{C_p}(A')$ lifting $\xi_A$.

Proof of part (ii). We define the canonical functor $\Psi\colon\Hilb^w_X\to \Def_{C_a}$ as follows: let $\xi_A=(X_A\subset\P_A(w))$ in $\Hilb^w_{C_p}(A)$ be an embedded deformation of $X\subset\P(w)$. Define $R_A=A[\tilde x_1,\dots,\tilde x_n]/I_{X_A}$, where generators $\tilde x_i$ are chosen so that $\P_A(w)=\Proj_A A[\tilde x_1,\dots,\tilde x_n]$, and $I_{X_A}$ is the ideal defining $X_A\subset\P_A(w)$, i.e.~$X_A=\Proj_A R_A$. By construction, $R_A$ is a flat deformation of $R(X,L)$. Indeed, the generators $\tilde x_i$ clearly extend $x_i$ because $X_A$ is embedded in $\P_A(w)$, and the relations and syzygies of $R_A$ lift those of $R(X,L)$ because $X_A$ is flat. Thus we define $\Psi(\xi_A)=\bar{\xi}_A$ in $\Def_{C_a}(A)$ by $\Spec_A R_A$. The central fibre is $\Spec R(X,L)$ and $\bar{\xi}_A$ is automatically flat. Moreover, $\Psi$ is functorial because $\Hilb^w_X$ is.

Formal\,smoothness\,is\,proved\,in\,a\,similar\,way\,to\,part\,(i). This\,time\,we\,have $\Hilb^w_X(k[\epsilon])=H^0(X,N_{X/\P(w)})$
and by assumption, $T^1_{C_a}=T^1_{C_a}(0)$ so we again have a surjection $d\Psi\colon\Hilb^w_X(k[\epsilon])\to\Def_{C_a}(k[\epsilon])$.

Let $\bar\xi_{A'}$ in $\Def_{C_a}(A')$ be an extension of $\Psi(\xi_A)=\bar\xi_A$. The obstruction to lifting $\bar\xi_{A'}$ to $\Hilb^w_X(A')$ lies in $H^1(N_{X/\P(w)})$. Now, $H^1(N_{X/\P(w)})\subset H^1(C'_a,N_{C'_a})$ and we know that $\bar\xi_A|_{C'_a}$ lifts to $\bar\xi_{A'}|_{C'_a}$ by assumption. Thus the obstruction vanishes, and $\xi_{A'}$ in $\Hilb^w_X(A')$ exists.
\qed
\end{proof}

We think that the following result is known to experts, but we could not find it in
the literature. We prove it as an application of Proposition \ref{prop!weighted-schlessinger-conditions}.

\begin{cor}\label{cor:abeliancone}
Let $X$ be an abelian variety of dimension $n \ge 2$ and $L$ an ample line bundle on $X$.
Then the affine cone $C_a=C_a(X,L)$ has only conical deformations.
\end{cor}

\begin{proof}
Since $\T_X \simeq \mathcal{O}_X^{\oplus n}$, we have $H^1(X, \mathcal{T}_X \otimes L^{\otimes k}) =0$ because $H^1(X, L^{\otimes k}) =0$ for any $k \neq 0$
by Serre duality, Kodaira vanishing and $n \ge 2$.
Hence $H^1(X, \mathcal{E}_L \otimes L^{\otimes k})=0$ and thus $T^1_{C_a}(k) =0$ for $k \neq 0$ by \ref{sec!computing-T1}.
Now apply Proposition \ref{prop!weighted-schlessinger-conditions}(ii).
\qed
\end{proof}
Indeed, embedded deformations of the projective cone over an abelian variety
were shown to be conical in \cite{sommese}. A recent preprint \cite[Cor.~8.7]{KK}
contains a proof of Corollary \ref{cor:abeliancone} which works in positive characteristic.

\newsubsection{Quasismooth K3 surfaces} It would be very interesting to generalise Theorem \ref{thm!main} to the case of affine cones over quasismooth K3 surfaces embedded in weighted projective space. Some applications of this are worked out in \cite{K3-transitions}, and we have already made some progress in this direction with Proposition \ref{prop!weighted-schlessinger-conditions}. We believe it is possible to further extend the Proposition to the quasismooth case. This motivates future work.

\providecommand{\bysame}{\leavevmode\hbox to3em{\hrulefill}\thinspace}
%
%

\bibliographystyle{amsalpha}
\bibliographymark{References}
\def\cprime{$'$}

\end{document}